\documentclass[12pt]{amsart}
\date{\today}

\usepackage[cp1251]{inputenc}         

\usepackage[ukrainian]{babel} 
\usepackage{amsmath,amsthm,amsfonts,amssymb,mathrsfs}

\usepackage{color}

\usepackage{hyperref}

\newtheorem{theorem}{Теорема}

\newtheorem{proposition}{Твердження}
\newtheorem{corollary}{Наслiдок}
\newtheorem{lemma}{Лема}
\theoremstyle{definition}
\newtheorem{example}{Приклад}

\begin{document}

\title[Про напiвгрупу $\textbf{ID}_{\infty}$]{Про напiвгрупу $\textbf{ID}_{\infty}$}

\author[O.~Гутік, A.~Савчук]{Олег Гутік, Анатолій Савчук}
\address{Механіко-математичний факультет, Львівський національний університет ім. Івана Франка, Університецька 1, Львів, 79000, Україна}
\email{o\_gutik@franko.lviv.ua, ovgutik@yahoo.com, asavchuk1@meta.ua}

\keywords{Semigroup of isometries, partial bijection, topological semigroup, compact, countably compact, feebly compact, discrete space, embedding. }

\subjclass[2010]{20M18, 20M20, 20M30, 22A15, 22A25, 54D30, 54D40, 54E52, 54H10}

\begin{abstract}
Досліджуємо напівгрупу $\textbf{\textsf{ID}}_{\infty}$ всіх часткових коскінченних ізометрій множини цілих чисел $\mathbb{Z}$. Доведено, що фактор-напівгрупа $\textbf{\textsf{ID}}_{\infty}/\mathfrak{C}_{\textsf{mg}}$ за мінімальною груповою конгруенцією $\mathfrak{C}_{\textsf{mg}}$ ізоморфна групі ${\textsf{Iso}}(\mathbb{Z})$ усіх ізометрій множини $\mathbb{Z}$, $\textbf{\textsf{ID}}_{\infty}$ є $F$-інверсною напівгрупою, а також, що напівгрупа $\textbf{\textsf{ID}}_{\infty}$ ізоморфна напівпрямому добутку ${\textsf{Iso}}(\mathbb{Z})\ltimes_\mathfrak{h}\mathscr{P}_{\!\infty}(\mathbb{Z})$ вільної напівґратки з одиницею $(\mathscr{P}_{\!\infty}(\mathbb{Z}),\cup)$ групою ${\textsf{Iso}}(\mathbb{Z})$. Знайдено достатні умови, за виконання яких, трансляційно неперевнна топологія на $\textbf{\textsf{ID}}_{\infty}$ є дискретною, а також побудовано недискретну гаусдорфову напівгрупову топологію на $\textbf{\textsf{ID}}_{\infty}$. Досліджуємо проблему ізоморфного занурення дискретної напівгрупи $\textbf{\textsf{ID}}_{\infty}$ у гаусдорфові топологічні напівгрупи близькі до компактних.

\medskip

\textbf{Oleg Gutik, Anatolii Savchuk,
 On the semigroup $\textbf{ID}_{\infty}$ }

We study the semigroup $\textbf{\textsf{ID}}_{\infty}$ of all partial isometries of the set of integers $\mathbb{Z}$. It is proved that the quotient semigroup $\textbf{\textsf{ID}}_{\infty}/\mathfrak{C}_{\textsf{mg}}$, where $\mathfrak{C}_{\textsf{mg}}$ is the minimum group congruence,  is isomorphic to the group ${\textsf{Iso}}(\mathbb{Z})$ of all isometries of $\mathbb{Z}$, $\textbf{\textsf{ID}}_{\infty}$ is an $F$-inverse semigroup, and $\textbf{\textsf{ID}}_{\infty}$ is isomorphic to the semidirect product ${\textsf{Iso}}(\mathbb{Z})\ltimes_\mathfrak{h}\mathscr{P}_{\!\infty}(\mathbb{Z})$ of the free semilattice with unit $(\mathscr{P}_{\!\infty}(\mathbb{Z}),\cup)$ by the group ${\textsf{Iso}}(\mathbb{Z})$. We give the sufficient conditions on a shift-continuous topology $\tau$ on $\textbf{\textsf{ID}}_{\infty}$ when $\tau$ is discrete. A non-discrete Hausdorff semigroup topology on $\textbf{\textsf{ID}}_{\infty}$ is constructed. Also, the problem of an embedding of the discrete semigroup $\textbf{\textsf{ID}}_{\infty}$ into Hausdorff compact-like topological semigroups is studied.
\end{abstract}

\maketitle


\section{Термінологія та означення}

Ми користуватимемося термінологією з \cite{Clifford-Preston-1961-1967, Engelking-1989, Lawson-1998, Petrich-1984, Ruppert-1984}.

Надалі у тексті потужність множини $A$ позначатимемо через $|A|$, перший не\-скін\-чен\-ний кардинал через $\omega$, і множину цілих чисел --- через $\mathbb{Z}$. Через $\operatorname{cl}_X(A)$ і $\operatorname{int}_X(A)$ позначатимемо \emph{замикання} та \emph{внутрішність} підмножини $A$ в топологічному просторі $X$.

Якщо визначене часткове відображення $\alpha\colon X\rightharpoonup Y$ з множини $X$ у множину $Y$, то через $\operatorname{dom}\alpha$ i $\operatorname{ran}\alpha$ будемо позначати його \emph{область визначення} й \emph{область значень}, відповідно, а через $(x)\alpha$ і $(A)\alpha$ --- образи елемента $x\in\operatorname{dom}\alpha$ та підмножини $A\subseteq\operatorname{dom}\alpha$ при частковому відображенні $\alpha$, відповідно. Часткове відображення $\alpha\colon X\rightharpoonup Y$ називається \emph{ко-скінченним}, якщо множини $X\setminus\operatorname{dom}\alpha$ та $Y\setminus\operatorname{ran}\alpha$ --- скінченні.

Рефлексивне, антисиметричне та транзитивне відношення на множині $X$ називається \emph{частковим порядком} на $X$. Множина $X$ із заданим на ній частковим порядком $\leqslant$ називається \emph{частково впорядкованою множиною} і позначається $(X,\leqslant)$.

Елемент $x$ частково впорядкованої множини $(X,\leqslant)$ називається:
\begin{itemize}
  \item \emph{максимальним} (\emph{мінімальним}) в $(X,\leqslant)$, якщо з відношення $x\leqslant y$ ($y\leqslant x$) в $(X,\leqslant)$ випливає рівність $x=y$;
  \item \emph{найбільшим} (\emph{найменшим}) в $(X,\leqslant)$, якщо $y\leqslant x$ ($x\leqslant y$) для всіх $y\in X$.
\end{itemize}

У випадку, якщо $(X,\leqslant)$~--- частково впорядкована множина і $x\leqslant y$, для деяких $x,y\in X$, то будемо говорити, що елементи $x$ та $y$~--- \emph{порівняльні} в $(X,\leqslant)$. Якщо ж для елементів $x\leqslant y$ не виконується жодне з відношень $x\leqslant y$ або $y\leqslant x$, то говоритимемо, що елементи $x$ та $y$ --- непорівняльні у частково впорядкованій множині $(X,\leqslant)$. Частковий порядок $\leqslant$ на $X$ називається \emph{лінійним}, якщо довільні два елементи в $(X,\leqslant)$ --- порівняльні.

Відображення $h\colon X\rightarrow Y$ з частково впорядкованої множини $(X,\leqslant)$ в частково впорядковану множину $(Y,\leqslant)$ називається \emph{монотонним}, якщо з $x\leqslant y$ випливає $(x)h\leqslant (y)h$.

Якщо $S$~--- напівгрупа, то її підмножина ідемпотентів позначається через $E(S)$.  Напівгрупа $S$ називається \emph{інверсною}, якщо для довільного її елемента $x$ існує єдиний елемент $x^{-1}\in S$ такий, що $xx^{-1}x=x$ та $x^{-1}xx^{-1}=x^{-1}$. В інверсній напівгрупі $S$ вищеозначений елемент $x^{-1}$ називається \emph{інверсним до} $x$. \emph{В'язка}~--- це напівгрупа ідемпотентів, а \emph{напівґратка}~--- це комутативна в'язка. Надалі через $(\mathscr{P}_{\!\infty}(\mathbb{R}),\cup)$ по\-зна\-ча\-ти\-ме\-мо \emph{вільну напівґратку} з одиницею над множиною дійсних чисел, тобто множину усіх скінченних (разом з порожньою) підмножин множини $\mathbb{R}$ з операцією об'єднання.

Відношення еквівалентності $\mathfrak{K}$ на напівгрупі $S$ називається \emph{конгруенцією}, якщо для елементів $a$ i $b$ напівгрупи $S$ з того, що виконується умова $(a,b)\in\mathfrak{K}$ випливає, що $(ca,cb), (ad,bd) \in\mathfrak{K}$, для всіх $c,d\in S$. Відношення $(a,b)\in\mathfrak{K}$ ми також будемо записувати $a\mathfrak{K}b$, і в цьому випадку говоритимемо, що \emph{елементи $a$ i $b$ є $\mathfrak{K}$-еквівалентними}.

Якщо $S$~--- напівгрупа, то на $E(S)$ визначено частковий порядок
\begin{equation*}
    e\leqslant f \quad \hbox{ тоді і лише тоді, коли } \quad ef=fe=e.
\end{equation*}
Так означений частковий порядок на $E(S)$ називається \emph{природним}.

Означимо відношення $\leqslant$ на інверсній напівгрупі $S$ так:
\begin{equation*}
    s\leqslant t \qquad \hbox{тоді і лише тоді, коли}\qquad s=te,
\end{equation*}
для деякого ідемпотента $e\in S$. Так означений частковий порядок називається \emph{при\-род\-ним част\-ковим порядком} на інверсній напівгрупі $S$~\cite{Lawson-1998}. Очевидно, що звуження природного часткового порядку $\leqslant$ на інверсній напівгрупі $S$ на її в'язку $E(S)$ є при\-род\-ним частковим порядком на $E(S)$. Інверсна напівгрупа $S$ називається \emph{факторизовною}, якщо для кожного елемента $s\in S$ існує елемент $g$ групи одиниць напівгрупи $S$ такий, що $s\leqslant g$ стосовно природного часткового порядку $\leqslant$ на $S$.

Відомо, що група ${\textsf{Iso}}(\mathbb{Z})$ ізометрій множини цілих чисел $\mathbb{Z}$ ізоморфна напівпрямому добутку $\mathbb{Z}(+)\rtimes\mathbb{Z}_2$ адитивної групи цілих чисел $\mathbb{Z}(+)$ циклічною групою другого порядку $\mathbb{Z}_2$.

Через $\textbf{\textsf{ID}}_{\infty}$ позначимо напівгрупу всіх часткових коскінченних ізометрій множини цілих чисел $\mathbb{Z}$. Напівгрупа $\textbf{\textsf{ID}}_{\infty}$ означена в праці Безущак \cite{Bezushchak-2004}, де описано її твірні та доведено, що вона має експоненціальний ріст. Зауважимо, що напівгрупа $\textbf{\textsf{ID}}_{\infty}$ є інверсною і, очевидно, є піднапівгрупою напівгрупи всіх часткових коскінченних бієкцій множини цілих чисел $\mathbb{Z}$, а елементи напівгрупи $\textbf{\textsf{ID}}_{\infty}$ --- це саме звуження ізометрій множини цілих чисел $\mathbb{Z}$ на коскінченні підмножини в розумінні Лоусона (див. \cite[с. 9]{Lawson-1998}). У праці \cite{Bezushchak-2008} описані відношення Ґріна та головні ідеали напівгрупи $\textbf{\textsf{ID}}_{\infty}$.

Підмножина $A$  топологічного простору $X$ називається:
\begin{itemize}
  \item \emph{щільною} в $X$, якщо $\operatorname{cl}_X(A)=X$;
  \item \emph{кощільною} в $X$, якщо $X\setminus A$~--- щільна в $X$;
  \item \emph{ніде не щільною} в $X$, якщо $\operatorname{cl}_X(A)$~--- кощільна в $X$;
  \item \emph{$F_\sigma$-множиною}, якщо $A$ є зліченним об'єднанням замкнених множин.
\end{itemize}

\emph{Компактифікацією Стоуна-Чеха} тихоновського простору $X$ називається компактний гаусдорфовий простір $\beta X$, який містить $X$ як щільний підпростір такий, що довільне неперервне відображення $f\colon X\to Y$ в компактний гаусдорфовий простір $Y$ продовжується до неперервного відображення $\overline{f}\colon \beta X\to Y$ \cite{Engelking-1989}.

Топологічний простір $X$ називається:
\begin{itemize}
  \item \emph{квазірегулярним}, якщо для довільної непорожньої відкритої підмножини $U$ в $X$ існує відкрита непорожня підмножина $V\subseteq U$ така, що $\operatorname{cl}_X(V)\subseteq U$;
  \item \emph{компактним}, якщо довільне відкрите покриття простору $X$ містить скінченне підпокриття;
  \item \emph{секвенціально компактним}, якщо довільна послідовність в $X$ містить збіжну підпослідовність;
  \item \emph{зліченно компактним}, якщо довільне зліченне відкрите покриття простору $X$ містить скінченне підпокриття;
  \item \emph{слабко компактним}, якщо довільна локально скінченна сім'я відкритих непорожніх підмножин в $X$ є скінченною \cite{Bagley-Connell-McKnight-Jr-1958};
  \item \emph{$d$-слабко компактним} або \textsf{DFCC}-\emph{простором}, якщо довільна дискретна сім'я відкритих непорожніх підмножин в $X$ є скінченною (див. \cite{Matveev-1998});
  \item \emph{псевдокомпактним}, якщо $X$ є цілком регулярним і кожна неперервна дійснозначна функція на $X$ є обмеженою;
  \item \emph{локально компактним}, якщо для кожної точки $x\in X$ існує відкритий окіл $U(x)$ точки $x$ в $X$ з компактним замиканням $\operatorname{cl}_X(U(x))$;
  \item \emph{повним за Чехом}, якщо $X$ є цілком регулярним і наріст $\beta X\setminus X$ є $F_\sigma$-множиною в $\beta X$;
  \item \emph{берівським}, якщо  для кожної послiдовностi $A_1,A_2,\ldots,A_n,\ldots$ нiде нещiльних множин з $X$ об’єднання $\bigcup_{i=1}^\infty A_i$ є кощiльною пiдмножиною в $X$;
  \item \emph{спадково берівським}, якщо довільна непорожна замкнена підмножина в $X$ є берівським простором.
\end{itemize}
За теоремою~3.10.22 з \cite{Engelking-1989}, цілком регулярний простір $X$ є слабко компактним тоді і лише тоді, коли $X$ є псевдокомпактним. Кожен компакний та кожен секвенціально компактний простір є зліченно компактним, кожен зліченно компактний простір є слабко компактним, а кожен слабко компактний простір є $d$-слабко компактним.

Топологічний простір $X$ із заданою на ньому напівгруповою операцією називається \emph{напівтопологічною} (\emph{топологічною}) \emph{напівгрупою}, якщо ця напівгрупова операція на $X$ є нарізно (сукупно) неперервною. Інверсна топологічна напівгрупа з неперервною інверсією називається \emph{топологічною інверсною напівгрупою}. Топологія $\tau$ на [інверсній] напівгрупі $S$ називається [\emph{інверсною}] \emph{напівгруповою}, якщо $(S,\tau)$ --- топологічна [інверсна] напівгрупа. Також, топологія $\tau$ на  напівгрупі $S$ називається \emph{трансляційно неперервною}, $(S,\tau)$ --- напівтопологічна напівгрупа

Вивчення напівгрупових недискретних тополологізацій  напівгруп починається з класичної праці Ебергарта та Селдена \cite{Eberhart-Selden-1969}, у якій доведено, що кожна напівгрупова гаусдорфова топологія на біциклічному моноїді $\mathscr{C}(p,q)$ є дискретною. У праці \cite{Bertman-West-1976} Бертман і Уест довели, що  кожна гаусдорфова трансляційно неперервна топологія на $\mathscr{C}(p,q)$ є також дискретною. У працях \cite{Fihel-Gutik-2011, Gutik-Maksymyk-2016a??} згадані результати були поширені на розширену біциклічну напівгрупу $\mathscr{C}_{\mathbb{Z}}$ та на інтерасоціативності бі\-цикліч\-ного моноїда. Також Тайманов у \cite{Taimanov-1973} побудував приклад напівгрупи, яка допускає лише дискретну напівгрупову гаусдорфову топологію та в \cite{Taimanov-1975} він визначив достатні ознаки недискретної напівгрупової топологізації комутативної напівгрупи. Напівгрупа $T$ називається {\em Таймановою}, якщо вона містить дві різні точки $0_T,\infty_T$ такі, що $xy=\infty_T$ для довільних різних точок $x,y\in T\setminus\{0_T,\infty_T\}$ і $xy=\infty_T$ у всіх інших випадках \cite{Gutik-2016}. У праці \cite{Gutik-2016} доведено, що довільна Тайманова напівгрупа має такі топологічні властивості: (i) кожна $T_1$-топологія з неперервними зсувами на $T$ є дискретною; (ii) $T$ замкнена в довільній $T_1$-топологічній напівгрупі, що містить $T$ як піднапівгрупу; (iii) кожен неізоморфний гомоморфний образ $Z$ напівгрупи $T$ є напівгрупою з нульовим множенням і, отже, є топологічною напівгрупою в довільній топології на $Z$. Дискретні та недискретні топологізації напівгруп перетворень за модулем берівського чи локально компактного простору вивчали в працях \cite{Bardyla-2016, Bardyla-Gutik-2016, Chuchman-Gutik-2010, Chuchman-Gutik-2011, Gutik-2015, Gutik-Maksymyk-2016??, Gutik-Pozdnyakova-2014, Gutik-Repovs-2011, Gutik-Repovs-2012}.

Проблема ізоморфного занурення напівгруп у гаусдорфові топологічні напівгрупи близькі до компактних досліджували в  \cite{Anderson-Hunter-Koch-1965, Banakh-Dimitrova-Gutik-2009, Banakh-Dimitrova-Gutik-2010, Bardyla-Gutik-2016, Chuchman-Gutik-2010, Chuchman-Gutik-2011, Guran-Gutik-Rav-Chu-2011, Gutik-Lawson-Repov-2009, Gutik-Maksymyk-2016a??, Gutik-Pavlyk-2005, Gutik-Pavlyk-Reiter-2009, Gutik-Repovs-2007}.

Ми доводимо, що фактор-напівгрупа $\textbf{\textsf{ID}}_{\infty}/\mathfrak{C}_{\textsf{mg}}$ за мінімальною груповою конгруенцією $\mathfrak{C}_{\textsf{mg}}$ ізоморфна групі ${\textsf{Iso}}(\mathbb{Z})$ усіх ізометрій множини $\mathbb{Z}$; напівгрупа $\textbf{\textsf{ID}}_{\infty}$ є $F$-інверсною напівгрупою, а також, що напівгрупа $\textbf{\textsf{ID}}_{\infty}$ ізоморфна напівпрямому добутку ${\textsf{Iso}}(\mathbb{Z})\ltimes_\mathfrak{h}\mathscr{P}_{\!\infty}(\mathbb{Z})$ вільної напівґратки з одиницею $(\mathscr{P}_{\!\infty}(\mathbb{Z}),\cup)$ групою ${\textsf{Iso}}(\mathbb{Z})$. Визначено достатні умови, за виконання яких трансляційно неперервна топологія на $\textbf{\textsf{ID}}_{\infty}$  є дискретною, а також побудовано недискретну гаусдорфову напівгрупову топологію на $\textbf{\textsf{ID}}_{\infty}$. Досліджується проб\-ле\-ма ізоморфного занурення дискретної напівгрупи $\textbf{\textsf{ID}}_{\infty}$ у гаусдорфові топологічні напівгрупи близькі до компактних.

\medskip

\section{Структурна теорема для напівгрупи $\textbf{\textsf{ID}}_{\infty}$}

Найменша групова конгруенція $\mathfrak{C}_{\textsf{mg}}$ на інверсній напівгрупі $S$ визначається так (див. \cite[III.5]{Petrich-1984}):
\begin{equation*}
    s\mathfrak{C}_{\textsf{mg}}t \hbox{~в~} S \quad \hbox{~тоді і лише тоді, коли існує ідемпотент~} \quad  e\in S \quad  \hbox{~такий, що~} \quad  es=et.
\end{equation*}

З означення напівгрупи $\textbf{\textsf{ID}}_{\infty}$ випливає, що для довільного елемента $\alpha$ напівгрупи $\textbf{\textsf{ID}}_{\infty}$ існує єдиний елемент $\gamma_\alpha$ групи одиниць $H(1)$ такий, що $\alpha\leqslant\gamma_\alpha$, а отже, означено відображення
\begin{equation}\label{eq-Gamma}
 \mathfrak{G}\colon \textbf{\textsf{ID}}_{\infty}\rightarrow H(1)\colon\alpha\mapsto\gamma_\alpha.
\end{equation}
З означення напівгрупи $\textbf{\textsf{ID}}_{\infty}$ випливає, що так означене відображення $\mathfrak{G}$ є сюр'єк\-тив\-ним гомоморфізмом, і тим більше для $\alpha,\beta\in \textbf{\textsf{ID}}_{\infty}$ маємо, що
\begin{equation*}
    \alpha\mathfrak{C}_{\textsf{mg}}\beta  \quad \hbox{~тоді і лише тоді, коли~} \quad  (\alpha) \mathfrak{G}=(\beta) \mathfrak{G}.
\end{equation*}

Отож, ми довели таку теорему

\begin{theorem}\label{theorem-2.1}
Фактор-напівгрупа $\textbf{\textsf{ID}}_{\infty}/\mathfrak{C}_{\textsf{mg}}$ ізоморфна групі \emph{${\textsf{Iso}}(\mathbb{Z})$} усіх ізометрій множини $\mathbb{Z}$, причому природний гомоморфізм \emph{$\mathfrak{C}_{\textsf{mg}}^{\natural}\colon \textbf{\textsf{ID}}_{\infty}\rightarrow {\textsf{Iso}}(\mathbb{Z})$} визначається за формулою \eqref{eq-Gamma}.
\end{theorem}

Нагадаємо, що інверсна напівгрупа $S$ називається \emph{$F$-інверсною}, якщо $\mathfrak{C}_{\textsf{mg}}$-клас $s_{\mathfrak{C}_{\textsf{mg}}}$ кожного елемента $s$  має найбільший елемент стосовно природного часткового порядку в $S$ \cite{McFadden-Carroll-1971}. Очевидно, що кожна $F$-інверсна напівгрупа містить одиницю.

З означення напівгрупи $\textbf{\textsf{ID}}_{\infty}$ випливає, що довільного елемента $\alpha$ напівгрупи $\textbf{\textsf{ID}}_{\infty}$ існує єдиний елемент $\gamma_\alpha$ групи одиниць $H(1)$ такий, що $\alpha\leqslant\gamma_\alpha$, а отже, виконується

\begin{corollary}\label{corollary-3.1}
$\textbf{\textsf{ID}}_{\infty}$ є $F$-інверсною напівгрупою.
\end{corollary}

Для довільного елемента $s$ інверсної напівгрупи $S$ позначимо
\begin{equation*}
{\downarrow}s=\{x\in S\colon x\leqslant s\},
\end{equation*}
 де $\leqslant$~--- природний частковий порядок на $S$.

Нехай $S$~--- довільна $F$-інверсна напівгрупа. Тоді для довільного елемента $s$ на\-пів\-гру\-пи $S$ через $e_s$ позначимо ідемпотент $ss^{-1}\in S$, через $t_s$~--- найбільший елемент стосовно природного часткового порядку на $S$ в $\mathfrak{C}_{\textsf{mg}}$-класі $s_{\mathfrak{C}_{\textsf{mg}}}$ елемента $s$, і нехай $T_S=\{t_s\colon s\in S\}$. Тоді напівгрупа $S$ є диз'юнктним об'єднанням множин ${\downarrow}t$, де $t\in T_S$ \cite{McFadden-Carroll-1971}.

Структура $F$-інверсних напівгруп викладена у  \cite{McFadden-Carroll-1971}, надалі ми далі використаємо такі два твердження для описання напівгрупи $\textbf{\textsf{ID}}_{\infty}$.

\begin{lemma}[{\cite[лема~3]{McFadden-Carroll-1971}}]\label{lemma-3.2}
Нехай $S$~--- $F$-інверсна напівгрупа з одиницею $1_S$. Тоді:
\begin{itemize}
  \item[$(i)$] $1_S$~--- одиниця напівґратки $E(S)$;
  \item[$(ii)$] множина $T_S$ з бінарною операцією
\begin{equation*}
    u\ast v=t_{uv}, \qquad u,v\in T_S,
\end{equation*}
  є групою з нейтральним елементом $1_S$, і $t^{-1}$ є оберненим до елемента $t$ в групі $(T_S,\ast)$;
  \item[$(iii)$] для кожного елемента $t\in T_S$ відображення $\mathfrak{F}_t\colon E(S)\rightarrow {\downarrow}e_t$, означене за формулою
\begin{equation*}
    (f)\mathfrak{F}_t=tft^{-1}, \qquad f\in E(S),
\end{equation*}
   є сюр'єктивним гомоморфізмом, причому $\mathfrak{F}_{1_S}$ є тотожним відображенням на $E(S)$;
   \item[$(iv)$] $(1_S)\mathfrak{F}_t=e_t$ й $(e_t)\mathfrak{F}_{t^{-1}}=e_{t^{-1}}$, для довільного елемента $t\in T_S$;
   \item[$(v)$] для довільних елементів $u,v\in S$ виконується рівність
   \begin{equation*}
   \left((1_S)\mathfrak{F}_u\right)\mathfrak{F}_v\cdot (f)\mathfrak{F}_{u\ast v}=\left((f)\mathfrak{F}_u\right)\mathfrak{F}_v, \qquad  \hbox{для довільного ідемпотента} \quad f\in S;
   \end{equation*}
   \item[$(vi)$] якщо $u,v\in T_S$, то
\begin{equation*}
    f\cdot(g)\mathfrak{F}_u\leqslant e_{u\ast v},
\end{equation*}
   для всіх ідемпотентів $f\leqslant e_u$ та $g\leqslant e_v$ напівгрупи $S$.
\end{itemize}
\end{lemma}

\begin{theorem}[{\cite[теорема~3]{McFadden-Carroll-1971}}]\label{theorem-3.3}
Нехай $S$~--- $F$-інверсна напівгрупа та $\mathscr{S}{=}\displaystyle\!\bigcup_{t\in T_S}\!\left({\downarrow}e_t{\times}\{t\}\right)$. Означимо на $\mathscr{S}$ бінарну операцію $\circ$ так: якщо $u,v\in T_S$, то для ідем\-по\-тен\-тів $f\leqslant e_u$ та $g\leqslant e_v$ приймемо
\begin{equation}\label{eq-circ}
    (f,u)\circ(g,v)=\left(f\cdot(g)\mathfrak{F}_u,u\ast v\right).
\end{equation}
Тоді $\circ$~--- напівгрупова операція на $\mathscr{S}$ і напівгрупа $\left(\mathscr{S},\circ\right)$ ізоморфна напівгрупі $S$ стосовно відображення $\mathfrak{H}\colon S\rightarrow \mathscr{S}\colon s\mapsto\left(ss^{-1},t_s\right)$.
\end{theorem}

Нехай $A$ та $B$~--- напівгрупи, $\operatorname{\textsf{End}}(B)$~--- напівгрупа ендоморфізмів напівгрупи $B$ і визначено гомоморфізм $\mathfrak{h}\colon A\rightarrow \operatorname{\textsf{End}}(B)\colon b\mapsto \mathfrak{h}_b$. Тоді множина $A\times B$ з бінарною операцією
\begin{equation*}
    (a_1,b_1)\cdot(a_2,b_2)=\left(a_1a_2,(b_1)\mathfrak{h}_{a_2}b_2\right), \qquad a_1,a_2\in A, \; b_1,b_2\in B
\end{equation*}
називається \emph{напівпрямим добутком} напівгрупи $A$ напівгрупою $B$ стосовно гомо\-мор\-фіз\-му $\mathfrak{h}$ і позначається $A\ltimes_\mathfrak{h}B$ \cite{Lawson-1998}. У цьому випадку кажуть, що визначена права дія напівгрупи $A$ на напівгрупі $B$ ендоморфізмів (гомоморфізмів). Зауважимо, що напівпрямий добуток інверсних напівгруп не завжди є інверсною напівгрупою (див. \cite[розділ~5.3]{Lawson-1998}).

\begin{lemma}\label{lemma-3.4}
Відображення $\mathfrak{h}\colon H(1)\rightarrow\operatorname{\textsf{End}}\left(E(\textbf{\textsf{ID}}_{\infty})\right) \colon \gamma\mapsto\mathfrak{h}_\gamma$, де $(\alpha)\mathfrak{h}_\gamma=\gamma^{-1}\alpha\gamma$~--- ав\-то\-мор\-фізм напівґратки $E(\textbf{\textsf{ID}}_{\infty})$, є гомоморфізмом, причому $\mathfrak{h}_1$~--- тотожний ав\-то\-мор\-фізм на\-пів\-ґрат\-ки $E(\textbf{\textsf{ID}}_{\infty})$.
\end{lemma}

\begin{proof}
Для довільних $\gamma\in H(1)$, $\varepsilon,\iota\in E(\textbf{\textsf{ID}}_{\infty})$ отримаємо, що
\begin{equation*}
    (\varepsilon\iota)\mathfrak{h}_\gamma=\gamma^{-1}\varepsilon\iota\gamma=\gamma^{-1}\varepsilon\gamma\gamma^{-1}\iota\gamma= (\varepsilon)\mathfrak{h}_\gamma(\iota)\mathfrak{h}_\gamma,
\end{equation*}
а отже, $\mathfrak{h}_\gamma$~--- гомоморфізм напівґратки $E(\textbf{\textsf{ID}}_{\infty})$. Оскільки для довільних $\gamma\in H(1)$ та $\varepsilon\in E(\textbf{\textsf{ID}}_{\infty})$ елемент $\gamma\varepsilon\gamma^{-1}$ є ідемпотентом напівгрупи $\textbf{\textsf{ID}}_{\infty}$ і
\begin{equation*}
(\gamma\varepsilon\gamma^{-1})\mathfrak{h}_\gamma=\gamma^{-1}(\gamma\varepsilon\gamma^{-1})\gamma=\varepsilon,
\end{equation*}
то гомоморфізм $\mathfrak{h}_\gamma$ --- сюр'єктивне відображення. Очевидно, що $\mathfrak{h}_1$~--- тотожне відображення напівґратки $E(\textbf{\textsf{ID}}_{\infty})$.

Припустимо, що $(\varepsilon)\mathfrak{h}_\gamma=(\iota)\mathfrak{h}_\gamma$, для деяких $\gamma\in H(1)$, $\varepsilon,\iota\in E(\textbf{\textsf{ID}}_{\infty})$. Оскільки $H(1)$~--- група одиниць напівгрупи $\textbf{\textsf{ID}}_{\infty}$, то з рівностей
\begin{equation*}
    \gamma^{-1}\varepsilon\gamma=(\varepsilon)\mathfrak{h}_\gamma=(\iota)\mathfrak{h}_\gamma=\gamma^{-1}\iota\gamma
\end{equation*}
випливає, що
\begin{equation*}
    \varepsilon=1\varepsilon 1=\gamma\gamma^{-1}\varepsilon\gamma\gamma^{-1}=\gamma\gamma^{-1}\iota\gamma\gamma^{-1}=1\iota 1=\iota,
\end{equation*}
а отже, $\mathfrak{h}_\gamma$~--- ав\-то\-мор\-фіз\-м напівґратки $E(\textbf{\textsf{ID}}_{\infty})$.

Зафіксуємо довільні $\gamma,\delta\in H(1)$. Тоді для довільного ідемпотента $\varepsilon\in \textbf{\textsf{ID}}_{\infty}$ маємо, що
\begin{equation*}
    (\varepsilon)\mathfrak{h}_{\gamma\delta}= (\gamma\delta)^{-1}\varepsilon\gamma\delta= \delta^{-1}\gamma^{-1}\varepsilon\gamma\delta= \delta^{-1}(\varepsilon)\mathfrak{h}_{\gamma}\delta= \left((\varepsilon)\mathfrak{h}_{\gamma}\right)\mathfrak{h}_\delta= (\varepsilon)\left(\mathfrak{h}_{\gamma}\cdot\mathfrak{h}_\delta\right),
\end{equation*}
а отже, так означене відображення $\mathfrak{h}\colon H(1)\rightarrow\operatorname{\textsf{End}}\left(E(\textbf{\textsf{ID}}_{\infty})\right)$ є гомоморфізмом.
\end{proof}

Наступна теорема описує структуру напівгрупи $\textbf{\textsf{ID}}_{\infty}$.

\begin{theorem}\label{theorem-3.5}
Напівгрупа $\textbf{\textsf{ID}}_{\infty}$ ізоморфна напівпрямому добутку \emph{${\textsf{Iso}}(\mathbb{Z})\ltimes_\mathfrak{h}\mathscr{P}_{\!\infty}(\mathbb{Z})$} вільної напівґратки з одиницею $(\mathscr{P}_{\!\infty}(\mathbb{Z}),\cup)$ групою \emph{${\textsf{Iso}}(\mathbb{Z})$} усіх ізометрій множини цілих чисел $\mathbb{Z}$.
\end{theorem}

\begin{proof}
Оскільки група одиниць $H(1)$ напівгрупи $\textbf{\textsf{ID}}_{\infty}$ ізоморфна групі ${\textsf{Iso}}(\mathbb{Z})$ усіх ізометрій множини $\mathbb{Z}$, то нам достатньо довести, що напівгрупа $\textbf{\textsf{ID}}_{\infty}$ ізо\-морф\-на напівпрямому добутку $H(1)\ltimes_\mathfrak{h}E(\textbf{\textsf{ID}}_{\infty})$ на\-пів\-ґрат\-ки $E(\textbf{\textsf{ID}}_{\infty})$ групою одиниць $H(1)$ напівгрупи $\textbf{\textsf{ID}}_{\infty}$ сто\-сов\-но гомоморфізму $\mathfrak{h}\colon H(1)\rightarrow\operatorname{\textsf{End}}\left(E(\textbf{\textsf{ID}}_{\infty})\right) \colon \gamma\mapsto\mathfrak{h}_\gamma$, де $(\alpha)\mathfrak{h}_\gamma=\gamma^{-1}\alpha\gamma$.

Означимо відображення $\mathfrak{T}\colon \textbf{\textsf{ID}}_{\infty}\rightarrow H(1)\ltimes_\mathfrak{h}E(\textbf{\textsf{ID}}_{\infty})$ за формулою
\begin{equation*}
    (\alpha)\mathfrak{T}=\left(\gamma_\alpha,\alpha^{-1}\alpha\right),
\end{equation*}
де елемент $\gamma_\alpha$ групи одиниць $H(1)$ напівгрупи $\textbf{\textsf{ID}}_{\infty}$, визначений формулою \eqref{eq-Gamma}. Оскільки для довільного елемента $\alpha$ напівгрупи $\textbf{\textsf{ID}}_{\infty}$ існує єдиний елемент $\gamma_\alpha$ групи одиниць $H(1)$ такий, що $\alpha\leqslant\gamma_\alpha$, то з наслідку~\ref{corollary-3.1} випливає, що відображення $\mathfrak{T}\colon \textbf{\textsf{ID}}_{\infty}\rightarrow H(1)\ltimes_\mathfrak{h}E(\textbf{\textsf{ID}}_{\infty})$ означене коректно, і воно є сюр'єктивним. При\-пус\-ти\-мо, що існують еле\-мен\-ти $\alpha$ та $\beta$ напівгрупи $\textbf{\textsf{ID}}_{\infty}$ такі, що $(\alpha)\mathfrak{T}=(\beta)\mathfrak{T}$. Тоді $\left(\gamma_\alpha,\alpha^{-1}\alpha\right)=\left(\gamma_\beta,\beta^{-1}\beta\right)$ і, використавши властивість, що для довільного елемента $\alpha$ напівгрупи $\textbf{\textsf{ID}}_{\infty}$ існує єдиний елемент $\gamma_\alpha$ групи одиниць $H(1)$ такий, що $\alpha\leqslant\gamma_\alpha$, i лему~1.4.6 з \cite{Lawson-1998}, отримуємо
\begin{equation*}
    \alpha=\gamma_\alpha\alpha^{-1}\alpha=\gamma_\beta\beta^{-1}\beta=\beta,
\end{equation*}
а отже, відображення $\mathfrak{T}\colon \textbf{\textsf{ID}}_{\infty}\rightarrow H(1)\ltimes_\mathfrak{h}E(\textbf{\textsf{ID}}_{\infty})$ є сюр'єктивним.

Нехай $\alpha$ та $\beta$ --- довільні елементи напівгрупи $\textbf{\textsf{ID}}_{\infty}$. Тоді з формули \eqref{eq-Gamma} і  тео\-реми~\ref{theorem-2.1} випливає, що $\alpha\beta\mathfrak{C}_{\textsf{mg}}\gamma_\alpha\gamma_\beta$, і оскільки $\gamma_\alpha,\gamma_\beta\in H(1)$, то отримуємо, що $\gamma_\alpha\gamma_\beta=\gamma_{\alpha\beta}$. Звідси вип\-ли\-ває, що
\begin{equation*}
\begin{split}
  (\alpha)\mathfrak{T}(\beta)\mathfrak{T} & = \left(\gamma_\alpha,\alpha^{-1}\alpha\right)\left(\gamma_\beta,\beta^{-1}\beta\right)= \\
    & =\big(\gamma_\alpha\gamma_\beta,\gamma_\beta^{-1}\alpha^{-1}\alpha\gamma_\beta\beta^{-1}\beta\big)=\\
    & = \big(\gamma_{\alpha\beta},\gamma_\beta^{-1}\alpha^{-1}\alpha\gamma_\beta\beta^{-1}\beta\big).
\end{split}
\end{equation*}
За лемою~\ref{lemma-3.4} відображення $\mathfrak{h}_\gamma\colon E(\textbf{\textsf{ID}}_{\infty})\rightarrow E(\textbf{\textsf{ID}}_{\infty})\colon\alpha\mapsto\gamma^{-1}\alpha\gamma$ є ав\-то\-мор\-фіз\-мом напівґратки $E(\textbf{\textsf{ID}}_{\infty})$, а отже отримуємо, що елемент $\gamma_\beta^{-1}\alpha^{-1}\alpha\gamma_\beta$ є ідем\-по\-тен\-том напівгрупи $\textbf{\textsf{ID}}_{\infty}$. Оскільки $\textbf{\textsf{ID}}_{\infty}$~--- інверсна напівгрупа, то для довільного елемента $\alpha$ напівгрупи $\textbf{\textsf{ID}}_{\infty}$ існує єдиний елемент $\gamma_\alpha$ групи одиниць $H(1)$ такий, що $\alpha\leqslant\gamma_\alpha$. З леми~1.4.6 з \cite{Lawson-1998} виливає, що  $\beta=\gamma_\beta\beta^{-1}\beta$, а отже,
\begin{equation*}
\begin{split}
  \gamma_\beta^{-1}\alpha^{-1}\alpha\gamma_\beta\beta^{-1}\beta & = \left(\gamma_\beta^{-1}\alpha^{-1}\alpha\gamma_\beta\right)\left(\beta^{-1}\beta\right)\left(\beta^{-1}\beta\right)=\\
    & = \left(\beta^{-1}\beta\gamma_\beta^{-1}\right)\left(\alpha^{-1}\alpha\right)\left(\gamma_\beta\beta^{-1}\beta\right)=\\
    & = \left(\gamma_\beta\beta^{-1}\beta\right)^{-1}\left(\alpha^{-1}\alpha\right)\left(\gamma_\beta\beta^{-1}\beta\right)=\\
    & = \beta^{-1}\left(\alpha^{-1}\alpha\right)\beta=\\
    & = \left(\beta^{-1}\alpha^{-1}\right)\left(\alpha\beta\right)=\\
    & = \left(\alpha\beta\right)^{-1}\left(\alpha\beta\right).
\end{split}
\end{equation*}
Отож, отримуємо
\begin{equation*}
    (\alpha\beta)\mathfrak{T}=\left(\gamma_{\alpha\beta},(\alpha\beta)^{-1}\alpha\beta\right)=(\alpha)\mathfrak{T}(\beta)\mathfrak{T},
\end{equation*}
а отже, відображення $\mathfrak{T}\colon \textbf{\textsf{ID}}_{\infty}\rightarrow H(1)\ltimes_\mathfrak{h}E(\textbf{\textsf{ID}}_{\infty})$ є гомоморфізмом, що і за\-вер\-шує доведення теореми.
\end{proof}

Позаяк група ${\textsf{Iso}}(\mathbb{Z})$ ізоморфна напівпрямому добутку $\mathbb{Z}(+)\rtimes\mathbb{Z}_2$, то з теореми~\ref{theorem-3.5} випливає наслідок.

\begin{corollary}\label{corollary-3.6}
Напівгрупа $\textbf{\textsf{ID}}_{\infty}$ ізоморфна напівпрямому добутку
\begin{equation*}
\left(\mathbb{Z}(+)\rtimes\mathbb{Z}_2\right)\ltimes_\mathfrak{h}\mathscr{P}_{\!\infty}(\mathbb{Z}).
\end{equation*}
\end{corollary}

Надалі нам буде потрібне таке твердження.

\begin{proposition}\label{proposition-3.7}
Для довільного $\alpha\in \textbf{\textsf{ID}}_{\infty}$ множина
\begin{equation*}
{\uparrow} \alpha=\left\{\beta\in \textbf{\textsf{ID}}_{\infty}\colon \alpha\leqslant\beta\right\}
\end{equation*}
скінченна.
\end{proposition}

\begin{proof}
З леми~1.4.6 \cite{Lawson-1998} випливає, що
\begin{equation*}
{\uparrow} \alpha=\left\{\beta\in \textbf{\textsf{ID}}_{\infty}\colon \alpha=\alpha\alpha^{-1}\beta\right\},
\end{equation*}
і, використавши те, що у напівгрупі $\textbf{\textsf{ID}}_{\infty}$ усі ідемпотенти є частковими тотожними відображеннями коскінченних у $\mathbb{Z}$ підмножин, отримуємо, що підмножина ${\uparrow} \alpha$ є скінченною в $\textbf{\textsf{ID}}_{\infty}$.
\end{proof}

\begin{proposition}\label{proposition-3.8}
Для довільних $\alpha,\beta\in \textbf{\textsf{ID}}_{\infty}$ множини
\begin{equation*}
  R(\alpha|\beta)=\left\{\chi\in \textbf{\textsf{ID}}_{\infty}\colon \alpha\chi=\beta\right\} \qquad \hbox{i} \qquad L(\alpha|\beta)=\left\{\chi\in \textbf{\textsf{ID}}_{\infty}\colon \chi\alpha=\beta\right\}
\end{equation*}
скінченні. Причому, якщо $R(\alpha|\beta)\neq\varnothing$   {$(L(\alpha|\beta)\neq\varnothing)$}, то $\left|R(\alpha|\beta)\cap H(1)\right|=1$   $(\left|L(\alpha|\beta)\cap H(1)\right|=1)$.
\end{proposition}

\begin{proof}
Зауважимо, очевидно, що $R(\alpha|\beta)$ є підмножиною множини
\begin{equation*}
  R=\left\{\chi\in \textbf{\textsf{ID}}_{\infty}\colon \alpha^{-1}\alpha\chi=\alpha^{-1}\beta\right\},
\end{equation*}
оскільки $\alpha^{-1}\alpha$ є ідемпотентом напівгрупи $\textbf{\textsf{ID}}_{\infty}$, то отримуємо, що $R\subseteq {\uparrow}\alpha^{-1}\beta$. Тоді за твердженням~\ref{proposition-3.7} отримуємо, що $R$~--- скінченна, а отже, $R(\alpha|\beta)$ --- скінченна підмножина в $\textbf{\textsf{ID}}_{\infty}$. Останнє твердження випливає з того, що $\textbf{\textsf{ID}}_{\infty}$ є $F$-напівгрупою. Доведення твердження у випадку множини $L(\alpha|\beta)$ є аналогічним.
\end{proof}


\medskip

\section{Про (напів)топологічну напівгрупу $\textbf{\textsf{ID}}_{\infty}$}

\begin{proposition}\label{proposition-4.1}
Нехай $\tau$ --- $T_1$-топологія на напівгрупі $\textbf{\textsf{ID}}_{\infty}$ стосовно якої ліві (праві) зсуви в $\left(\textbf{\textsf{ID}}_{\infty},\tau\right)$ є неперервними відображеннями та топологічний простір $\left(\textbf{\textsf{ID}}_{\infty},\tau\right)$ містить ізольовану точку. Тоді група одиниць $H(1)$ є дискретним підпростором в $\left(\textbf{\textsf{ID}}_{\infty},\tau\right)$.
\end{proposition}

\begin{proof}
Припустимо, що ліві зсуви в $\left(\textbf{\textsf{ID}}_{\infty},\tau\right)$ є неперервними відображеннями. Нехай $\alpha_0$~--- ізольована точка в $\left(\textbf{\textsf{ID}}_{\infty},\tau\right)$. Тоді з неперервності лівих зсувів у $\left(\textbf{\textsf{ID}}_{\infty},\tau\right)$ випливає, що множина ${\uparrow} \alpha_0$ --- відкрито-замкнена, як повний прообраз відкрито-замкненої множини при неперервному лівому зсуві на ідемпотент $\alpha_0\alpha_0^{-1}$. Позаяк $\tau$ --- $T_1$-топологія на напівгрупі $\textbf{\textsf{ID}}_{\infty}$, то за твердженням~\ref{proposition-3.8} група одиниць $H(1)$ напівгрупи $\textbf{\textsf{ID}}_{\infty}$ містить ізольовану точку в $\left(\textbf{\textsf{ID}}_{\infty},\tau\right)$. Але в кожній групі рівняння $ax=b$ має єдиний розв'язок і множина $\textbf{\textsf{ID}}_{\infty}\setminus H(1)$ є двобічним ідеалом в напівгрупі $\textbf{\textsf{ID}}_{\infty}$, то з неперервності лівих зсувів у $\left(\textbf{\textsf{ID}}_{\infty},\tau\right)$ випливає, що кожен елемент групи одиниць $H(1)$ є ізольованою точкою в просторі $\left(\textbf{\textsf{ID}}_{\infty},\tau\right)$.

Для правих зсувів доведення аналогічне.
\end{proof}

\begin{theorem}\label{theorem-4.2-1}
Нехай $\tau$ --- берівська $T_1$-топологія на напівгрупі $\textbf{\textsf{ID}}_{\infty}$, стосовно якої ліві (праві) зсуви в $\left(\textbf{\textsf{ID}}_{\infty},\tau\right)$ є неперервними відображеннями. Тоді група одиниць $H(1)$ є дискретним підпростором в $\left(\textbf{\textsf{ID}}_{\infty},\tau\right)$.
\end{theorem}

\begin{proof}
Припустимо, що ліві зсуви в $\left(\textbf{\textsf{ID}}_{\infty},\tau\right)$ є неперервними відображеннями. Позаяк напівгрупа $\textbf{\textsf{ID}}_{\infty}$ зліченна, то з беровості $T_1$-простору $\left(\textbf{\textsf{ID}}_{\infty},\tau\right)$ випливає, що хоча б один елемент сім'ї $\left\{\left\{\alpha\right\}\colon \alpha\in \textbf{\textsf{ID}}_{\infty}\right\}$ має непорожню внутрішність в $\left(\textbf{\textsf{ID}}_{\infty},\tau\right)$, а отже, простір $\left(\textbf{\textsf{ID}}_{\infty},\tau\right)$ містить ізольовану точку. Далі скористаємося  тверджен\-ням~\ref{proposition-4.1}.
\end{proof}

Кажуть, що топологія $\tau$ на напівгрупі $S$ є \emph{ліво} (\emph{право}) $E$-\emph{берівською}, якщо для довільного ідемпотента $e\in S$ підпростір $eS$ ($Se$) в $S$ є берівським.

\begin{theorem}\label{theorem-4.2}
Кожна  ліво (право) $E$-берівська $T_1$-топологія $\tau$ на напівгрупі $\textbf{\textsf{ID}}_{\infty}$, стосовно якої праві (ліві) зсуви в $\left(\textbf{\textsf{ID}}_{\infty},\tau\right)$ є неперервними відображеннями, дискретна.
\end{theorem}

\begin{proof}
Нехай $\tau$ --- ліво $E$-берівська $T_1$-топологія на напівгрупі $\textbf{\textsf{ID}}_{\infty}$, стосовно якої ліві (праві) зсуви в $\left(\textbf{\textsf{ID}}_{\infty},\tau\right)$ є неперервними відображеннями. Оскільки $\textbf{\textsf{ID}}_{\infty}$ --- інверсна напівгрупа, то $\alpha \textbf{\textsf{ID}}_{\infty}=\alpha\alpha^{-1} \textbf{\textsf{ID}}_{\infty}$ для довільного елемента $\alpha\in \textbf{\textsf{ID}}_{\infty}$, а отже, $\alpha \textbf{\textsf{ID}}_{\infty}$~--- берівський підпростір в $\left(\textbf{\textsf{ID}}_{\infty},\tau\right)$. Тоді з беровості $T_1$-простору $\alpha \textbf{\textsf{ID}}_{\infty}$ випливає, що хоча б один елемент сім'ї $\left\{\left\{\beta\right\}\colon \beta\in \alpha \textbf{\textsf{ID}}_{\infty}\right\}$ має непорожню внутрішність в $\alpha \textbf{\textsf{ID}}_{\infty}$, а отже, простір $\alpha \textbf{\textsf{ID}}_{\infty}$ містить ізольовану точку.  Нехай $\alpha_0$~--- ізольована точ\-ка в $\alpha \textbf{\textsf{ID}}_{\infty}$ і $\alpha\beta=\alpha_0$ для деякого $\beta\in \textbf{\textsf{ID}}_{\infty}$. Тоді з неперервності правих зсувів у $\left(\textbf{\textsf{ID}}_{\infty},\tau\right)$ та твердження~\ref{proposition-3.8} випливає, що множина
\begin{equation*}
 \left\{\chi\in \textbf{\textsf{ID}}_{\infty}\colon \chi\beta=\alpha_0\right\}
\end{equation*}
є скінченною та відкритою в $\left(\textbf{\textsf{ID}}_{\infty},\tau\right)$ як повний прообраз відкритої множини при неперервному правому зсуві на елемент $\beta$, і, крім того, вона містить елемент $\alpha$. Звідки випливає, що $\alpha$ --- ізольована точка в $\left(\textbf{\textsf{ID}}_{\infty},\tau\right)$. З довільності вибору елемента $\alpha\in \textbf{\textsf{ID}}_{\infty}$ випливає, що усі точки простору $\left(\textbf{\textsf{ID}}_{\infty},\tau\right)$  ізольовані.
\end{proof}

Позаяк в гаусдорфовій напівтопологічній напівгрупі $S$ множини $eS$ i $Se$ --- замк\-не\-ні для довільного ідемпотента $e\in S$, то з теореми~\ref{theorem-4.2} випливає такий наслідок

\begin{corollary}\label{corollary-4.3}
Кожна гаусдорфова трансляційно неперервна спадково берівська топологія $\tau$ на напівгрупі $\textbf{\textsf{ID}}_{\infty}$   є дискретною.
\end{corollary}

Оскільки кожен гаусдорфовий локально компактний топологічний простір є повним за Чехом, а кожен повний за Чехом є спадково берівським (див. \cite{Engelking-1989}), то з наслідку~\ref{corollary-4.3} випливає наслідок~\ref{corollary-4.4}.

\begin{corollary}\label{corollary-4.4}
Кожна гаусдорфова трансляційно неперервна повна за Чехом (а отже, і локально компактна) топологія $\tau$ на напівгрупі $\textbf{\textsf{ID}}_{\infty}$  є дискретною.
\end{corollary}

З наступного прикладу випливає, що на напівгрупі $\textbf{\textsf{ID}}_{\infty}$ існує недискретна неберівська гаусдорфова топологія $\tau_\textsf{NB}$ така, що $\left(\textbf{\textsf{ID}}_{\infty},\tau_\textsf{NB}\right)$ є топологічною напівгрупою.

\begin{example}\label{example-4.5}
Відомо, що група ізометрій ${\textsf{Iso}}(\mathbb{T}^1)$ одиничного кола
\begin{equation*}
 \mathbb{T}^1=\left\{z\in\mathbb{C}\colon |z|=1\right\}
\end{equation*}
на комплексній площині ізоморфна напівпрямому добутку $\mathbb{T}^1\rtimes\mathbb{Z}_2$ та відображення
\begin{equation*}
\theta\colon \mathbb{Z}(+)\rtimes\mathbb{Z}_2\rightarrow \mathbb{T}^1\rtimes\mathbb{Z}_2\colon (z,a)\mapsto(e^{iz},a)
\end{equation*}
є ізоморфним (алгебричним) зануренням групи ${\textsf{Iso}}(\mathbb{Z})$ в групу ${\textsf{Iso}}(\mathbb{T}^1)$. Нехай на одиничному колі $\mathbb{T}^1$ задано компактну топологію, індуковану з $\mathbb{C}$, де на $\mathbb{C}$ визначена звичайна евклідова топологія, а на групі $\mathbb{Z}_2$ визначена дискретна топологія. Тоді $\mathbb{T}^1\rtimes\mathbb{Z}_2$ з топологією добутку є компактною топологічною групою (див. \cite[приклад 6.22]{Roelcke-Dierolf-1982}), яка індукує на підгрупі ${\textsf{Iso}}(\mathbb{Z})$ недискретну групову топологію.

Нехай на напівґратці $\mathscr{P}_{\!\infty}(\mathbb{Z})$ задана дискретна топологія. Тоді за теоремою~2.10 з \cite[том~1, с.~67]{Carruth-Hildebrant-Koch-1983-1986} напівгрупа $\left(\mathbb{Z}(+)\rtimes\mathbb{Z}_2\right)\ltimes_\mathfrak{h}\mathscr{P}_{\!\infty}(\mathbb{Z})$ з топологією добутку є топологічною напівгрупою, яка, очевидно, не є дискретним простором.
\end{example}

\begin{lemma}\label{lemma-4.6-0}
Якщо $A$~--- дискретний щільний підпростір $T_1$-топологічного простору $X$, то $A$~--- відкритий підпростір в $X$.
\end{lemma}

\begin{proof}
Припустимо протилежне: існує точка $x\in A$ така, що кожен її відкритий окіл $U(x)$ перетинає множину $X\setminus A$. Зафіксуємо довільний відкритий окіл $U_0(x)$ точ\-ки $x$ в топологічному просторі $X$ такий, що $U_0(x)\cap A=\{x\}$. Тоді $U_0(x)$ є відкритим околом деякої точки $y\in U_0(x)\cap X\setminus A$ в топологічному просторі $X$, а отже, $U_0(x)$ містить нескінченну кількість точок множини $A$, що суперечить вибору околу $U_0(x)$. З отриманої суперечності випливає твердження леми.
\end{proof}

\begin{theorem}\label{theorem-4.6}
Нехай дискретна напівгрупа $\textbf{\textsf{ID}}_{\infty}$ є щільною піднапівгрупою $T_1$-напівтопологічної напівгрупи  $S$ й $I=S\setminus \textbf{\textsf{ID}}_{\infty}\neq\varnothing$. Тоді $I$ є двобічним ідеалом в~$S$.
\end{theorem}

\begin{proof}
З леми~\ref{lemma-4.6-0} випливає, що $\textbf{\textsf{ID}}_{\infty}$ є відкритим підпростором в $S$.

Зафіксуємо довільний елемент $y\in I$, якщо $x\cdot y=z\notin I$ для деякого елемента $x\in \textbf{\textsf{ID}}_{\infty}$, то існує відкритий окіл $U(y)$ точки $y$ в топологічному просторі $S$ такий, що  $\{x\}\cdot U(y)=\{z\}\subset \textbf{\textsf{ID}}_{\infty}$. Окіл $U(y)$ містить нескінченну кількість елементів напівгрупи  $\textbf{\textsf{ID}}_{\infty}$, що суперечить твердженню~\ref{proposition-3.8}. З отриманого протиріччя випливає, що $x\cdot y\in I$ для всіх $x\in \textbf{\textsf{ID}}_{\infty}$ and $y\in I$. Доведення твердження, що $y\cdot x\in I$ для всіх $x\in \textbf{\textsf{ID}}_{\infty}$ та $y\in I$ є аналогічним.

Припустимо протилежне: $x\cdot y=w\notin I$, для деяких $x,y\in I$. Тоді $w\in \textbf{\textsf{ID}}_{\infty}$ і з нарізної неперервності напівгрупової операції в $S$ випливає, що існують відкриті околи $U(x)$ та $U(y)$ точок $x$ та $y$ в просторі $S$, відповідно, такі, що $\{x\}\cdot U(y)=\{w\}$ and $U(x)\cdot \{y\}=\{w\}$. Однак обидва околи $U(x)$ і $U(y)$ містять нескінченну кількість елементів напівгрупи $\textbf{\textsf{ID}}_{\infty}$, а отже, обидві рівності $\{x\}\cdot U(y)=\{w\}$ й $U(x)\cdot \{y\}=\{w\}$ суперечать попередній частині доведення теореми, оскільки $\{x\}\cdot \left(U(y)\cap \textbf{\textsf{ID}}_{\infty}\right)\subseteq I$. З отриманого протиріччя випливає, що $x\cdot y\in I$.
\end{proof}

\begin{proposition}\label{proposition-4.7}
Нехай гаусдорфова топологічна напівгрупа $S$ містить дискретну напівгрупу $\textbf{\textsf{ID}}_{\infty}$ як щільну піднапівгрупу. Тоді для довільного $c\in \textbf{\textsf{ID}}_{\infty}$ множина
\begin{equation*}
    D_c=\left\{(x,y)\in \textbf{\textsf{ID}}_{\infty}\times \textbf{\textsf{ID}}_{\infty}\colon x y=c\right\}
\end{equation*}
вікрито-замкнена в $S\times S$.
\end{proposition}

\begin{proof}
З леми~\ref{lemma-4.6-0} випливає, що $\textbf{\textsf{ID}}_{\infty}$ є відкритим підпростором в $S$. Тоді з неперервності напівгрупової операції в напівгрупі $S$ отримуємо, що $D_c$ --- відкрита підмножина в просторі $S\times S$ для довільного елемента $c\in \textbf{\textsf{ID}}_{\infty}$.

Припустимо, що існує такий елемент $c\in \textbf{\textsf{ID}}_{\infty}$, що $D_c$ є незамкненою підмножиною в $S\times S$. Тоді існує точка накопичення $(a,b)\in S\times S$ множини  $D_c$. З неперервності напівгрупової операції в $S$ випливає, що $a\cdot b=c$. Але  $\textbf{\textsf{ID}}_{\infty}\times \textbf{\textsf{ID}}_{\infty}$ є дискретним підпростором в  $S\times S$, а отже, за теоремою~\ref{theorem-4.6} точки $a$ і $b$ належать до двобічного ідеалу $I=S\setminus \textbf{\textsf{ID}}_{\infty}$, а звідси випливає, що добуток $a\cdot b\in S\setminus \textbf{\textsf{ID}}_{\infty}$ не може дорівнювати елементові $c$.
\end{proof}

\begin{theorem}\label{theorem-4.8}
Якщо гаусдорфова топологічна напівгрупа $S$ містить напівгрупу $\textbf{\textsf{ID}}_{\infty}$ як щільну дискретну піднапівгрупу, то квадрат $S\times S$ не є слабко компактним простором.
\end{theorem}

\begin{proof}
З твердженням~\ref{proposition-4.7} для довільного елемента $c\in \textbf{\textsf{ID}}_{\infty}$ квадрат $S\times S$ містить відкрито-замкнений дискретний підпростір $D_c$. У випадку, коли $c$ є одиницею групи одиниць напівгрупи $\textbf{\textsf{ID}}_{\infty}$, то множина $D_c$ містить нескінченну підмножину $\left\{\left(x,x^{-1}\right)\colon x\in H(1)\right\}$, а отже, множина $D_c$ є нескінченною. Звідси випливає, що простір $S\times S$ не є слабко компактним.
\end{proof}

З іншого боку, кожен зліченно компактний простір є слабко компактним і за теоремою~3.10.4 з \cite{Engelking-1989} замкнений підпростір зліченно компактного простору є знову зліченно компактним, то з теореми~\ref{theorem-4.8} випливає наслідок~\ref{corollary-4.9}.

\begin{corollary}\label{corollary-4.9}
Якщо гаусдорфова топологічна напівгрупа $S$ містить дискретну напівгрупу $\textbf{\textsf{ID}}_{\infty}$, то її квадрат $S\times S$ не є зліченно компактним простором.
\end{corollary}

Відомо, що компактність і секвенціальна компактність зберігається скінченними добутками (див. \cite[розділ 3]{Engelking-1989}), а отже, з наслідку~\ref{corollary-4.9} випливають такі два наслідки

\begin{corollary}\label{corollary-4.10}
Дискретна напівгрупа $\textbf{\textsf{ID}}_{\infty}$ не занурюється топологічно ізоморфно в жодну гаусдорфову компактну топологічну напівгрупу.
\end{corollary}

\begin{corollary}\label{corollary-4.11}
Дискретна напівгрупа $\textbf{\textsf{ID}}_{\infty}$ не занурюється топологічно ізоморфно в жодну гаусдорфову секвенціально компактну топологічну напівгрупу.
\end{corollary}

\begin{theorem}\label{theorem-4.13}
Якщо гаусдорфова топологічна напівгрупа $S$ містить напівгрупу $\textbf{\textsf{ID}}_{\infty}$ з ізольованою точкою в $\textbf{\textsf{ID}}_{\infty}$, то квадрат $S\times S$ не є зліченно компактним простором.
\end{theorem}

\begin{proof}
Якщо напівгрупа $\textbf{\textsf{ID}}_{\infty}$ містить ізольовану точку, то за твердженням~\ref{proposition-4.1} група одиниць $H(1)$ є дискретним підпростором в $\textbf{\textsf{ID}}_{\infty}$. Тоді за твердженням~2.3.3 з \cite{Engelking-1989} отримаємо, що $\operatorname{cl}_{S\times S}(H(1)\times H(1))=\operatorname{cl}_{S}(H(1))\times \operatorname{cl}_{S}(H(1))$, а тоді з теореми~3.10.4 з \cite{Engelking-1989} випливає, що $\operatorname{cl}_{S\times S}(H(1)\times H(1))$~--- зліченно компактний простір. Однак $\operatorname{cl}_{S\times S}(H(1)\times H(1))$~--- замкнена піднапівгрупа в $S\times S$ (див. \cite[т.~1, с.~9--10]{Carruth-Hildebrant-Koch-1983-1986}), то з аналогічних міркувань (як і в доведенні твердження~\ref{proposition-4.7} і теореми~\ref{theorem-4.8}) отримуємо, що $\operatorname{cl}_{S\times S}(H(1)\times H(1))$ не є слабко компактним підпростором в $S\times S$, а отже, за теоремою~3.10.4 з \cite{Engelking-1989} простір $S\times S$ не є зліченно компактним.
\end{proof}

\begin{theorem}\label{theorem-4.15}
Слабко компактна квазі-регулярна $T_1$-топологічна інверсна напівгрупа не містить напівгрупу $\textbf{\textsf{ID}}_{\infty}$ як щільну піднапівгрупу.
\end{theorem}

\begin{proof}
Припустимо протилежне: існує слабко компактна квазі-регулярна $T_1$-топологічна інверсна напівгрупа $S$, яка містить напівгрупу $\textbf{\textsf{ID}}_{\infty}$ як щільну піднапівгрупу. Тоді з теореми~2.8 монографії \cite{Haworth-McCoy-1977} випливає, що простір напівгрупи $S$ є берівським, а отже, хоча б один елемент сім'ї $\mathscr{U}=\left\{s\colon s\in \textbf{\textsf{ID}}_{\infty}\right\}\cup\left\{S\setminus \textbf{\textsf{ID}}_{\infty}\right\}$ має непорожню внутрішність. Позаяк усі точки множини $S\setminus \textbf{\textsf{ID}}_{\infty}$ є точками дотику до множини $\textbf{\textsf{ID}}_{\infty}$ у просторі $S$, то $\operatorname{int}_S\left(S\setminus \textbf{\textsf{ID}}_{\infty}\right)=\varnothing$, а отже, напівгрупа $\textbf{\textsf{ID}}_{\infty}$ містить ізольовану точку в просторі $S$. Тоді за твердженням~\ref{proposition-4.1} група одиниць $H(1)$ напівгрупи $\textbf{\textsf{ID}}_{\infty}$ є дискретним підпростором у напівгрупі $\textbf{\textsf{ID}}_{\infty}$, а отже, і у просторі $S$. Припустимо, що група одиниць $H(1)$ напівгрупи $\textbf{\textsf{ID}}_{\infty}$ не є замкненим підпростором у топологічній інверсній напівгрупі $S$. Зафіксуємо довільний елемент $s\in S$ такий, що $s\in \operatorname{cl}_S\left(H(1)\right)\setminus H(1)$. Доведемо, що $ss^{-1}\neq 1\neq s^{-1}s$. Нехай $U(1)$~--- довільний відкритий окіл одиниці $1$ напівгрупи $\textbf{\textsf{ID}}_{\infty}$ у топологічній інверсній напівгрупі $S$ такий, що $U(1)\cap \textbf{\textsf{ID}}_{\infty}=\{1\}$. Однак $S$~--- топологічна інверсна напівгрупа, а отже,  існує відкритий окіл $V(s)$ елемента $s$ у просторі $S$ такий, що
\begin{equation*}
V(s)\cdot \left(V(s)\right)^{-1}\cup \left(V(s)\right)^{-1}\cdot V(s)\subseteq U(1).
\end{equation*}
Але окіл $V(s)$ елемента $s$ містить нескінченну кількість елементів групи одиниць $H(1)$ напівгрупи $\textbf{\textsf{ID}}_{\infty}$, то
\begin{equation*}
\left(V(s)\cdot \left(V(s)\right)^{-1}\cup \left(V(s)\right)^{-1}\cdot V(s)\right)\cap \textbf{\textsf{ID}}_{\infty}\neq \{1\},
\end{equation*}
а це суперечить вибору околу $U(1)$. З отриманої суперечності випливає, що група одиниць $H(1)$ напівгрупи $\textbf{\textsf{ID}}_{\infty}$ є замкненим підпростором у топологічній інверсній напівгрупі $S$, і більше того, група одиниць $H(1)$ напівгрупи $\textbf{\textsf{ID}}_{\infty}$ є групою одиниць напівгрупи $S$.

Далі зауважимо, що група одиниць $H(1)$ напівгрупи $\textbf{\textsf{ID}}_{\infty}$ є відкритим під\-прос\-то\-ром в просторі~$S$. Справді, з  наведених міркувань випливає, що напівгрупа $S$ містить ізольовану точку $s_0\in \textbf{\textsf{ID}}_{\infty}$ в просторі $S$. Позаяк $S$~--- топологічна ін\-верс\-на напівгрупа, то повний прообраз $(\{s_0\})\rho_{s_0}^{-1}$ стосовно правого зсуву $\rho_{s_0}\colon S\rightarrow S: s\mapsto s\cdot s_0$ є відкритою підмножиною в $S$. З означення напівгрупи $\textbf{\textsf{ID}}_{\infty}$ випливає, що  $(\{s_0\})\rho_{s_0}^{-1}\cap H(1)=\{1\}$. Доведемо, шо множина $(\{s_0\})\rho_{s_0}^{-1}$ скінченна. Припустимо протилежне. Нехай  $(\{s_0\})\rho_{s_0}^{-1}$ --- нескінченна множина в $S$. Тоді з тверджен\-ня~\ref{proposition-3.8} випливає, що відкрита підмножина $(\{s_0\})\rho_{s_0}^{-1}$ містить нескінченну кількість елементів з наросту $S\setminus \textbf{\textsf{ID}}_{\infty}$, а отже, вона є відкритим околом деякої точки $x\in S\setminus \textbf{\textsf{ID}}_{\infty}$. Але довільний відкритий окіл точки $x\in S\setminus \textbf{\textsf{ID}}_{\infty}$ міс\-тить нескінченну кількість елементів напівгрупи $\textbf{\textsf{ID}}_{\infty}$, а це суперечить твердженню~\ref{proposition-3.8}, оскільки $\left((\{s_0\})\rho_{s_0}^{-1}\right)\rho_{s_0}=\left\{s_0\right\}$. Отож, ми отримали, що множина $(\{s_0\})\rho_{s_0}^{-1}$ відкрита та скінченна, і, крім того, вона містить одиницю $1$ групи одиниць $H(1)$. Отже, одиниця $1$ групи одиниць $H(1)$ напівгрупи $\textbf{\textsf{ID}}_{\infty}$ є ізольованою точкою в просторі $S$. Зафіксуємо довільний елемент $x_0$ групи одиниць $H(1)$. Позаяк $S$~--- топологічна інверсна напівгрупа, то повний прообраз $(\{1\})\rho_{x_0^{-1}}^{-1}$ стосовно правого зсуву $\rho_{x_0^{-1}}\colon S\rightarrow S: s\mapsto s\cdot x_0^{-1}$ є відкритою підмножиною в $S$. Тоді з означення напівгрупи $\textbf{\textsf{ID}}_{\infty}$ випливає, що $(\{1\})\rho_{x_0^{-1}}^{-1}\cap \textbf{\textsf{ID}}_{\infty}=\left\{x_0\right\}$. Використовуючи твердження~\ref{proposition-3.8}, отримуємо, що $(\{1\})\rho_{x_0^{-1}}^{-1}\cap S=\left\{x_0\right\}$. Отож, група одиниць $H(1)$ напівгрупи $\textbf{\textsf{ID}}_{\infty}$ є відкрито-замкненим дискретним підпростором у топологічній інверсній напівгрупі $S$, а це суперечить слабкій ком\-пакт\-нос\-ті прос\-то\-ру~$S$. З отриманого протиріччя випливає твердження теореми.
\end{proof}

\begin{theorem}\label{theorem-4.16}
Напівгрупа $\textbf{\textsf{ID}}_{\infty}$ не занурюється в жодну зліченно компактну $T_3$-топологічну інверсну напівгрупу.
\end{theorem}

\begin{proof}
Припустимо, що існує зліченно компактна $T_3$-топологічна інверсна напівгрупа $S$, яка містить напівгрупу $\textbf{\textsf{ID}}_{\infty}$. Тоді з теорем~2.1.6 і 3.10.4 з \cite{Engelking-1989} випливає, що замикання $\operatorname{cl}_S\left(\textbf{\textsf{ID}}_{\infty}\right)$ є зліченно компактним $T_3$-простором, а з твердження~II.2 з \cite{Eberhart-Selden-1969}, що $\operatorname{cl}_S\left(\textbf{\textsf{ID}}_{\infty}\right)$~--- топологічна інверсна напівгрупа. Отже, $T_3$-топологічна інверсна напівгрупа $\operatorname{cl}_S\left(\textbf{\textsf{ID}}_{\infty}\right)$ містить щільну напівгрупу, а це суперечить теоремі~\ref{theorem-4.15}. З отриманого протиріччя випливає твердження теореми.
\end{proof}

\section*{Подяка}

Автори висловлюють подяку С. Бардилі, О. Равському та рецензенту за корисні коментарі та зауваження.


\end{document}